\documentclass[letterpaper, 10 pt, conference]{ieeeconf}  

\IEEEoverridecommandlockouts                              

\usepackage{amssymb}

\usepackage{amsmath,latexsym}

\usepackage{amsfonts}
\newtheorem{assumption}{Hypothesis}

\newtheorem{remark}{Remark}[section]
\newtheorem{definition}{Definition}[section]

\newtheorem{theorem}{Theorem}[section]
\newtheorem{lemma}[theorem]{Lemma}
\newtheorem{proposition}[theorem]{Proposition}
\newtheorem{corollary}[theorem]{Corollary}
\newtheorem{example}[theorem]{Example}

\def\ds{\displaystyle}

\def\bega{\begin{array}}
\def\enda{\end{array}}

\def\bepmatrix{\begin{pmatrix}}
\def\enpmatrix{\end{pmatrix}}

\def\bel{\begin{equation}\label}
\def\eeq{\end{equation}}
\newcommand\ee{\end{equation}}
\def\benl{\begin{equation*}}
\def\eenl{\end{equation*}}

\def\forall{\hbox{for all }~}
\def\be{\begin{equation}}
\def\beq{\begin{equation}}
\def\bel{\begin{equation}\label}
\def\eeq{\end{equation}}

\newcommand\ba{\begin{array}}
\newcommand\ea{\end{array}}

\def\begi{\begin{itemize}}
\def\endi{\end{itemize}}

\def\pr{\partial}

\def\forall{\hbox{for all}~}

\def\R{I\!\!R}

\def\R{I\!\!R}
\newcommand{\cR}{\mathbb{R}}
\newcommand{\cN}{\mathbb{N}}

\newcommand{\ddt}{\frac{\rm d}{{\rm d} t} }

\def\C{\mathcal{C}}

\def\F{\mathcal{F}}

\def\L{\mathcal{L}}

\def\R{\mathcal{R}}

\def\U{\mathcal{U}}

\def\ub{\bar{u}}

\def\xb{\bar{x}}

\def\zb{\bar{z}}

\def\xib{\bar{\xi}}

\def\zetab{\bar{\zeta}}

\def\uh{\hat{u}}

\def\zetab{\bar\zeta}
\def\bel{\begin{equation}\label}
\def\eeq{\end{equation}}

\def\ubf{{\bf{u}}}
\def\vbf{{\bf v}}
\def\ybf{{\bf y}}

\title{\LARGE \bf
On optimal control problems
with impulsive commutative dynamics
}

\author{M. Soledad Aronna$^*$, Franco Rampazzo$^\dagger$\thanks{
This article will appear in the Proceedings of the 52nd IEEE Conference on Decision and Control, 2013.
This work is supported by the European Union under the 7th Framework Programme FP7-PEOPLE-2010-ITN -  Grant agreement 264735-SADCO.}
\thanks{$^*$$^\dagger$ M.S. Aronna and F. Rampazzo are with the Dipartimento di Matematica, Universit\`a di Padova, Padova  35121, Italy
 {\tt\small aronna@math.unipd.it, rampazzo@math.unipd.it}}%
}

\usepackage[margin=1.8cm]{geometry}

\begin{document}

\maketitle
\thispagestyle{empty}
\pagestyle{empty}

\begin{abstract}

We consider control systems governed by nonlinear O.D.E.'s that are  affine in the time-derivative du/dt of the control u.  The latter is allowed to be an integrable, possibly of unbounded variation function, which gives the system an impulsive character.  As is well-known, the corresponding Cauchy problem cannot   be interpreted in terms of Schwartz distributions, even in the commutative case. A robust notion of solution already proposed in the literature is here adopted and slightly generalized to the case where an ordinary, bounded, control is present in the dynamics as well.  For a  problem in the Mayer form we then investigate the question whether this notion of solution provides a ``proper extension" of the standard problem with absolutely continuous controls u. Furthermore, we show that this impulsive problem  is  a variational limit  of problems corresponding to controls u with bounded variation.

 \end{abstract}


\section{Introduction and basic notation}

Consider the control system
\begin{align}
\label{E}&\dot x = \tilde{f}(x,u,v) + \ds\sum_{\alpha=1}^m  \tilde g_\alpha(x,u)\dot u^\alpha, \\
\label{E0}&x(a)=\xb,\quad u(a)=\ub,
\end{align}
where $v:[a,b] \to V \subset \cR^l$ is a {standard bounded control} while  $u:[a,b] \to U \subseteq \cR^m$ is  an $\L^1-$function, which we refer to  as the {\em impulsive control.} The presence of the derivative $\dot u$ on the right hand-side  raises the issue of the definition of a (possibly discontinuous) solution  $x:[a,b] \to \cR^n.$
Several applications of this type of system are known, e.g. in mechanics, biology and economics. In optimal control theory impulses arise as soon as the control is unbounded and the cost lacks coercivity properties.
It is well-known that an approach based on Schwartz distributions cannot work (see e.g. \cite{Haj85}), this fact marking a crucial difference with the case when the vector fields $\tilde g_\alpha$ are constant.
However, an  appropriate  concept of solution for \eqref{E}-\eqref{E0} has been proposed in the late eighties (see e.g. \cite{Bre87,BreRam91,Dyk94}), also in connection with optimal control problems.
(Notice that  one cannot exploit  the notions of solutions  utilized for $u$  e.g. in \cite{BreRam88,Mil89,SilVin96}, for the controls are allowed to have {\it unbounded} variation).
In this paper we adopt and slightly extend the notion in \cite{BreRam91},   calling  it {\it pointwise defined solution} (shortly: {\it p.d. solution}\footnote{Actually, we call it {\em limit solution} in the subsequent articles (see \cite{AroRam13a,AroRam13b})}).
We begin by stating and partially proving elementary properties of p.d. solutions, like existence, uniqueness and continuous dependence on the data.  Afterwards, we focus on the question whether a Mayer type  optimal control  problem on the interval $[a,b],$
\be
\label{infL1}
\inf_{(u,v)\in \L^1\times L^1}\psi(x( b), u(b)),
\ee
 is in fact a {\it proper extension} of the standard problem
 \be
 \label{infAC}
 \inf_{(u,v)\in AC\times L^1}\psi(x( b), u(b)),
 \ee
where $\L^1$ and $L^1$ stand for the ``set of Lebesgue integrable functions"  (on [a,b])  and its quotient set, respectively; while $AC$ means ``absolutely continuous".

Loosely speaking, a {\it proper extension} of a minimum problem is a new problem in which the old one is embedded, in such a way  that the domain of the original problem is (somehow) dense in the new domain and the two problems have the same infimum value.

Our motivation  to study proper extensions of \eqref{infAC} comes mainly  from the need of  giving a physically acceptable meaning to  typical investigations for optimal control problems involving p.d. solutions. An instance is represented by necessary conditions for optimality. Indeed, in order that such necessary conditions are of practical use  one should rule out  the occurrence of  {\em Lavrentiev-like phenomena,} namely the fact that the infimum value of the extended problem is strictly less then that of the original system. Another instance that makes the search for proper extensions reasonable is dynamic programming and  its PDE expression, the Hamilton-Jacobi equation (see Section \ref{SecLast}).
Of course, the case where terminal constraints are imposed on the trajectories is of great  interest both for necessary conditions and dynamic programming. This case, which poses non-trivial additional difficulties, is investigated  in \cite{AroRam13c}.

The paper is organized as follows: in the remaining part of the present Section we introduce the notation and state some general  preliminary results. In Section \ref{SecCauchy} we present the definition and basic results on Cauchy problems involving p.d. solutions.
In Section \ref{SecProper} it is shown that the minimum problem with p.d. solutions is in fact a proper extension of the standard problem. In Section \ref{SecBV} we prove that the minimum problem with p.d. solutions is also the limit when $K\to+ \infty$ of the (impulse) problems corresponding  to $u$ with total variation bounded by $K$. In Section \ref{SecLast}, we propose some final considerations concerning existence of minima and the Hamilton-Jacobi equation for the problem \eqref{infL1}.

\subsection{Notation and preliminaries}

Let $[a,b]$ be a real interval and $E \subseteq \cR^d.$
$\L^1([a,b];E)$ will denote  the space of Lebesgue integrable functions defined on $[a,b]$ and having  values in $E.$  We shall use $L^1([a,b];E)$ to denote the corresponding set of equivalence classes, and $AC([a,b];E)$ to denote the space of absolutely continuous functions defined on $[a,b]$ with values in $E.$

Let us extend the functions $\tilde f,\tilde g_\alpha$, $\alpha =1,\dots,m$  to vector fields $f,g_\alpha$  on $\cR^{n+m}$ by setting
\bel{fieldsgf}
f:= f^j\frac{\partial}{\partial x^j},\qquad g_\alpha:= g_{\alpha}^j \frac{\partial}{\partial x^j} +     \frac{\partial}{\partial z^\alpha}\,,
\eeq
where $\left(\frac{\partial}{\partial x^1},\dots,\frac{\partial}{\partial x^n}, \frac{\partial}{\partial z^1},\dots, \frac{\partial}{\partial z^m}\right)$ is the canonical basis of $\cR^{n+m}$ and  we have adopted the {\em summation convention }over  repeated indexes. When not otherwise specified, Latin indexes run from $1$ to $n,$ while Greek indexes run from $1$ to $m.$

The hypothesis below is assumed throughout  the article.

\begin{assumption}[Commutativity]
\label{AssComm}
For every pair $\alpha,\beta,$
\bel{gicomm}
{[{g}_\alpha,{g}_\beta]=0,}
\ee
where $[{g}_\alpha,{g}_\beta] $ denotes the {\em Lie bracket}
of  $g_\alpha$ and $g_\beta,$ that in coordinates is defined by $[{g}_\alpha,{g}_\beta] := Dg_\beta\, g_\alpha - Dg_\alpha \, g_\beta.$
(Notice, in particular,  that the last $m$ components of $[{g}_\alpha,{g}_\beta]$ are zero.)
\end{assumption}

\begin{remark}
While this commutativity assumption is not a generic hypothesis, we impose it here motivated by the following reasons: 1) the scalar case and some mechanical applications are covered, and 2) we can ensure uniqueness of the solution of the impulsive Cauchy problem.
\end{remark}

Besides Hypothesis \ref{AssComm} we shall assume the following:

\begin{assumption}
\label{H1}
\begin{itemize}
 \item[(i)] $V \subseteq \cR^l$ is compact.
 \item[(ii)] For every $v \in V$, ${f}(\cdot,\cdot,v):\cR^{n+m}\to\cR^{n+m}$ is locally Lipschitz continuous and, for every $(x,u)\in\cR^{n+m}$ one has that ${f}(x,u,\cdot): V \to\cR^{n+m}$ is continuous.
 \item[(iii)] There exists $M>0$ such that $|f(x,u,v)| \leq  M(1+|(x,u)|),$ for every $(x,u) \in \cR^{n+m},$ uniformly in $v \in V.$
 \item[(iv)] The vector fields $g_\alpha:\cR^{n+m}\to\cR^{n+m}$ are of class $\C^1$ and there exists $N>0$ such that $|g_\alpha(x,u)| \leq  N(1+|(x,u)|),$ for every $(x,u) \in \cR^{n+m}.$
\end{itemize}
\end{assumption}

We observe that the sublinearity in (iii) and (iv) can be replaced by other conditions guaranteeing existence of the integral trajectories.

Let $h$ be a locally Lipschitz vector field on a $\C^1-$manifold $M$, and let $m\in M$. Whenever the solution to
\be
\ddt x(s) = h(x(s)),\quad h(0)=m
\ee
is defined on a interval $I$ containing $0$,   we use  ${\rm exp}({t h})(m)$ to denote the value of this  solution at time $t,$ for every $t\in I.$
We remark that the identification ${\rm exp}({h})={\rm exp}({1h})$
is consistent with this definition.

\section{The Cauchy problem}\label{SecCauchy}
Let us introduce a change of coordinates $\phi$ in the space $\cR^{n+m}$ that -thanks to Hypothesis \ref{AssComm}- simultaneously  transforms the vector fields $g_\alpha$ into constant vector fields.

 \subsection{A crucial change of coordinates}

Let  $\mathrm{Pr}:\cR^{n} \times \cR^m\rightarrow \cR^n$ denote the {canonical projection} on the first factor,
$\mathrm{Pr}(x,z) := x,$
and let the function  $\varphi:\cR^{n+m}\to \cR^n$ be defined by
\benl
\varphi(x,z) :=
\mathrm{Pr} \Big( \exp \left({- z_{\alpha} g_{\alpha}}\Big)
(x,z)\right).
\eenl
 Let us consider the map $\phi:\cR^{n+m}\rightarrow \cR^{n+m}$ defined by
\benl
\phi(x,z) := (\varphi(x,z),z).
\eenl
It  is straightforward to prove the following result:

\begin{lemma}
Assume that the vector fields $g_1,\dots,g_m$ are of class $\C^r,$ with $r\geq 1.$ Then $\phi$ is a $\C^r$-diffeomorphism of $\cR^{n+m}$ onto itself and, for every $(\xi,\zeta)\in \cR^{n+m},$ one has
\bel{phiinv}
\phi^{-1}(\xi,\zeta) = (\varphi(\xi,-\zeta), \zeta).
\eeq
\end{lemma}

The $\C^r$-diffeomorphism  $\phi$  induces a $\C^{r-1}$-diffeomorphism $D\phi$ on the tangent bundle. For each $\alpha=1,\hdots,m,$ let us set
\be\begin{array}{l}
F(\xi,\zeta,v):= D\phi (x,z) \, f(x,z,v),
\\
G_\alpha(\xi,\zeta):= D\phi (x,z) \, g_{\alpha}(x,z).
\end{array}
\ee

\begin{lemma}
\label{flowbox}
For every $i=1,\dots, n,$ $\alpha=1,\dots, m,$
\be
F =\left(\frac{\pr \varphi^i}{\pr x^j} \tilde{f}^j \right) \frac{\pr}{\pr x^i},\quad {G}_{\alpha}
= \frac{\pr }{\pr z^\alpha},
\ee
where we have set $\varphi=(\varphi^1,\dots,\varphi^n).$
\end{lemma}

\begin{remark}
The proof of Lemma \ref{flowbox} (see \cite[Lemma 2.1]{BreRam91} for details) is in fact a direct consequence of the  {\em Simultaneous Flow-Box Theorem}  (see e.g.  \cite{Lang}).
\end{remark}

Notice that the last $m$ components of $F$ are zero.  Therefore, in the new coordinates $(\xi,\zeta),$ the control system \eqref{E} turns into  the simpler form
\be
\label{xieq}
 \dot\xi(t) = \tilde{ F}(\xi(t),u(t),v(t)).
\ee
From now on we assume that the data are such that the Cauchy problem for \eqref{xieq} has a unique solution defined on $[a,b],$ for each $u\in AC([a,b];\cR^m),$ $v\in L^1([a,b];V).$ For instance, one can verify that this property holds true as soon as condition (iv) in Hypothesis \ref{H1} is replaced by (iv') below, which implies that $D \phi$ is globally bounded,
\vspace{1pt}
\begin{itemize}
\item[(iv')] $g_\alpha$ and $\C^1$ are globally Lipschitz.
\end{itemize}

\vspace{1pt}

Lemma \ref{equivC1} below concerns relations between the solutions of the control systems in both systems of coordinates.

Since we are going to exploit the diffeomorphism $\phi:\cR^{n+m} \to \cR^{n+m} $ it is convenient to embed \eqref{E}-\eqref{E0} in the $n+m$-dimensional Cauchy problem
\bel{AS}\left\{\begin{array}{l}
\vspace{1pt}
\begin{pmatrix}
\dot{x}\\
\dot{z}
\end{pmatrix}
= f(x,z,v)+   {g}_\alpha(x,z) \dot  u^\alpha,\\
\begin{pmatrix}
{x}\\{z}
\end{pmatrix}(a)= \begin{pmatrix} \xb \\ \zb \end{pmatrix}.\end{array}\right.
\ee
 Recall that the vector fields $f$ and $g_\alpha,$ are defined in $\cR^{n+m}\times V$ and $\cR^{n+m},$ respectively.
When $u\in AC([a,b];\cR^m),$ for every $(\xb,\zb) \in \cR^{m+n}$ and $v \in L^1([a,b];V),$  there exists a unique solution to \eqref{AS} in the interval $[a,b].$ We let  $(x,z)(\xb,\zb,u,v)(\cdot)$ denote this solution.

We shall also consider  the  Cauchy problem
\bel{TS}\left\{\begin{array}{l}
\vspace{1pt}
\begin{pmatrix}
\dot{\xi}\\
\dot{\zeta}
\end{pmatrix}
= F(\xi,\zeta,v)+   G_\alpha  \dot u^\alpha,\\
\begin{pmatrix}
{\xi}\\{\zeta}
\end{pmatrix}(a)= \begin{pmatrix} \xib \\ \zetab \end{pmatrix}
.\end{array}\right.
\ee
When $u\in AC([a,b];\cR^m),$ there exists a unique solution to \eqref{TS} in $[a,b].$ We let  $(\xi,\zeta)(\xib,\zetab,u,v)(\cdot)$ denote this solution.

The essential difference between the two systems relies on the fact that the {\em vector fields  $G_\alpha$ are constant}.  This allows us to give a notion of solution for  \eqref{TS} also for merely  integrable controls $u.$ Indeed, it is natural to set
\benl
\zeta(t) := \zetab +u(t) - u(a),
\eenl
 for all $t \in [a,b]$  and to let $\xi$ be  the Carath\'eodory  solution of the Cauchy problem $\dot\xi = F(\xi,\zeta,v),\,\,\xi(a) = \xib.$

When $u\in AC([a,b];\cR^m)$ the relation between the two systems is described in Lemma \ref{equivC1} below.
Let $(\xi,\zeta)(\xib,\zetab,u,v)(\cdot)$ denote  the unique solution of \eqref{TS} associated with $(\xib,\zetab)\in \cR^{n+m}$ and $(u,v) \in AC([a,b];\cR^m) \times L^1([a,b];V).$

\begin{lemma}
Let us consider $(\xb,\zb)\in \cR^{n+m}$ and controls
\label{equivC1}$u\in AC([a,b];\cR^m),$ $v \in L^1([a,b];V).$
Then,
\bel{xieta}
(\xi,\zeta)(\xib,\zetab,u,v)(t) = \phi\Big((x,z)(\xb,\zb,u,v)(t)\Big), \ee
\forall $t\in [a,b],$ where $(\xib,\zetab) := \phi(\xb,\zb).$
\end{lemma}

The latter result is a straightforward consequence of the definition of $F$ and $G_\alpha.$

\subsection {Ponitwise defined solutions}

Throughout the paper we shall assume that $U$ is an {\it impulse domain}:
\begin{definition}
Let $U\subseteq\cR^m.$ $U$ is called an {\em impulse domain} if, for every bounded interval $I \subset \cR,$ for each  function $u\in\L^1(I;U)$ and for every $t \in I,$ there exists a sequence $\{u_k\} \subset AC(I;U)$ such that $\|u_k-u\|_1\to 0$ and $u_k(t) \to u(t),$ when $n \to \infty.$
\end{definition}

Examples of impulse domains are:
\begin{itemize}
\item $U=\bar \Omega,$ with $\Omega$ a bounded, open, connected subset with Lipschitz boundary;
\item an embedded differentiable submanifold of $\cR^m$.
 \item a convex subset $U\subseteq\cR^m$ .
 \end{itemize}

\begin{definition}
Consider an initial data $\xb \in \cR^n$ and let  $(u,v)\in\L^1([a,b];U)\times L^1([a,b];V).$
We say that a map $x:[a,b]\to \cR^{n}$ is an {\em pointwise defined solution} (shortly  {\em p.d. solution}) of the Cauchy problem \eqref{E}-\eqref{E0} if, for every $t\in [a,b],$ the following conditions are met:
\begin{itemize}
\item[(i)] there exists a sequence $\{ u_k \} \subset AC([a,b];U)$ such that
$u_k\to u$ in $L^1([a,b];U),$ $u_k(a) \to u(a),$ $u_k(t) \to u(t),$ when  $k \to \infty,$ and  ;
\item[(ii)] for each $k \in \cN,$ there exists a (Carath\'eodory) solution
$x_k:[a,b]\to\cR^{n}$ of  \eqref{E}-\eqref{E0} corresponding to the control $(u_k,v)$ and the initial condition $(\xb,\ub := u(a));$
\item[(iii)] the sequence $\{ x_k \}$  has uniformly bounded values and converges to $x$  in $L^1([a,b];\cR^{n})$ and, moreover, $\ds \lim_{k \to \infty} x_k(t) = x(t).$
\end{itemize}
\end{definition}

 \begin{remark}
When $u$ is absolutely continuous, the  notion of e.d solution  is equivalent to  the standard concept of Carath\'eodory solution. Moreover, in \cite{AroRam13a}, we show that the notion of p.d. solution is quite general and when controls $u\in BV$ it coincides with the most  known concepts of solution, even in the generic case when the Lie brackets do not vanish.
\end{remark}

\begin{theorem}[Existence, uniqueness, representation]
\label{ex-un}
For every $\xb \in \cR^{n},$ and every control pair $(u,v)\in\L^1([a,b];U)\times L^1([a,b];V),$ there exists a unique p.d. solution of the Cauchy problem \eqref{E}-\eqref{E0} defined on $[a,b],$ where we have set $\ub := u(a).$ We shall use $x(\xb,u,v)(\cdot)$ to denote this solution.
Moreover, setting $\xib := \varphi(\xb,u(a)),$ one has
\be
x(\xb,u,v)(t) = \varphi(\xi(t),-u(t)),
\ee
for all $t\in [a,b],$ where $\xi(\cdot):=\xi(\xib,u,v)(\cdot)$ is the Carath\'eodory solution of the Cauchy problem
\bel{xiCO}
\dot\xi=F(\xi,u,v),\quad\xi(a) = \xib.
\ee
\end{theorem}

To prove this theorem, which extends an analogous result in \cite{BreRam91} where $f$ did not depend on the standard control $v,$ we shall make use of the following result.

 \begin{lemma}[see \cite{AroRam13b}]
 \label{cont-depAC}
 The following assertions hold true:
\begin{itemize}
\item[(i)] For $r>0$ and $K\subseteq U$ compact, there exists a compact subset $K'\subset \cR^{n},$ such that the trajectories $x(\bar x,u,v)(\cdot)$ have values in $K',$ whenever we consider $\xb \in B_r(0),$ $u\in AC([a,b];K)$ and $v\in L^1([a,b];V).$
\item[(ii)] For each $r$ and $K$ as in  (ii), there exists a constant $M>0$ such that, for every $t \in[a,b],$ for all $\xb_1, \xb_2 \in B_r(0),$ for all $u_1,u_2 \in AC([a,b];K)$ and for every $v \in L^1([a,b];V),$ one has
\benl
\begin{split}
|x_1(t)-x_2(t)|& +   \|x_1-x_2\|_1 \leq\\
& M\Big[ |\xb_1-\xb_2|+ |u_1(a)-u_2(a)| \\
& + |u_1(t)-u_2(t)|+\|u_1-u_2\|_1
\Big].
\end{split}
\eenl
where $x_1 := x(\xb_1,u_1,v_1),$ $x_2 := x(\xb_2,u_2,v_2).$
\end{itemize}

\end{lemma}

\begin{proof} {\it (of Theorem \ref{ex-un}) }
Set  $\ub:= u(a),$ $(\xib,\zetab) := \phi(\xb,\ub), $  $\zeta(\cdot):= \bar\zeta +u(\cdot)-\ub,$  and let $\xi$ be the solution of the differential equation \eqref{xieq} with initial condition $\xi(a) = \xib.$ Observe that $\zetab := \ub$ and hence, $\zeta \equiv u.$

Define $(x,z) := \phi^{-1}\circ(\xi,\zeta).$ Let us show that $(x,z)$ is a p.d. solution of \eqref{AS}. Choose $t \in [a,b]$ and a sequence of absolutely continuous controls $u_k: [a,b] \to U$ converging to $u$ in the $L^1$ topology and verifying $u_k(a) \to \ub,$ $u_k(t) \to u(t)$  when $k\to \infty.$
Since $u$ is bounded it is not restrictive to assume that the functions $\{ u_k \}$ have equibounded values.  Let $(\xi_k,\zeta_k)$ be the corresponding solutions to \eqref{TS} and set
\be
\label{xkdef}
(x_k,z_k):=  \phi^{-1}\circ(\xi_k,\zeta_k).
\ee
Note, in particular, that  the paths $(\xi_k,\zeta_k)$ and $(x_k,z_k)$ are equibounded.
Then
\benl
\begin{split}
\|(x,z)-&(x_k,z_k)\|_1 \\
&= \| \phi^{-1}\circ(\xi,\zeta) -
 \phi^{-1}\circ(\xi_k,\zeta_k)\|_1 \rightarrow_{k \to \infty} 0,
 \end{split}
 \eenl
as the map
 $\phi^{-1}$ is Lipschitz continuous on compact sets.
 Moreover, since $\zeta_k(t) = \zetab + u_k(t) - u_k(a),$ one has $\zeta_k(t) \to \zeta(t).$  Therefore, in view of \eqref{xkdef} and since $\xi_k \to \xi$ uniformly, $(x_k(t),z_k(t)) \to (x(t),z(t)).$
 This concludes the part concerning existence and representation of a solution.

In order to prove uniqueness, let $x^1(\cdot)$ and $x^2(\cdot)$ be solutions of \eqref{E}-\eqref{E0} both associated with the same data $\xb \in \cR^n,$ $(u,v) \in \L^1 \times L^1$ and where $\ub:=u(a).$  Assume by contradiction that there exists $t \in [a,b]$ such that
 $x^1(t)\neq x^2(t)$. According to the definition of p.d. solution there exist sequences  $\{ u^1_k\}_{k\in\mathbb{N}}$, $\{u^2_k\}_{k\in\mathbb{N}}$  in $AC ([a,b];U)$ such that, for $i=1,2$, one has
 \be\ba{c}
u^i_k(a) \to u(a),\quad u^i_k(t) \to u(t), \\
\|u^i_k - u\|_1\to 0,\quad \|x(\xb,u^i_k,v)-x^i\|_1\to 0.
 \ea
\eeq
Hence, by Lemma \ref{cont-depAC} above, we have,
\benl
\ba{c}
 |x^1_k(t)-x^2_k(t)| \leq M
\left( |u_k^1(a) - u_k^2(a)| \right.\\
\qquad\qquad\left.+ |u_k^1(t)-u_k^2(t)|+ \|u^1_k-u^2_k\|_1\right) \to 0.
\ea
\eenl
Therefore,
$
 |x^1(t)-x^2(t)|
 = \lim_{k\to\infty}
 |x^1_k(t)-x^2_k(t)|  = 0,
$
which is a contradiction. The proof is concluded.
\end{proof}
 
 \vspace{2pt}
 
Let us give below a toy example of a p.d. solution corresponding to a discontinuous $u$ with unbounded variation.
\begin{example}
Let us consider the differential equation
\be
\label{eqex}
\dot{x} = xv+ x\dot{u},\quad x(0)=\xb,
\ee
on the interval $[0,1],$ with
\benl
v(t):=
\left\{
\ba{cl}
1, & \text{for } t\in[0,1/2[,\\
0, & \text{for } t\in [1/2,1].
\ea
\right.
\eenl
Observe that, if $u \in AC([0,1];\cR),$ then, for any $[a,b]\subseteq [1/2,1],$ the associated Carath\'eodory solution of \eqref{eqex} verifies
\be
\label{forex}
x(t)=x(a)e^{u(t) - u(a)}.
\ee
Consider now the $\L^1-$control
\benl
u(t):=
\left\{
\ba{cl}
(-1)^{k+1}, & \text{for } t\in[1-\frac{1}{k},1-\frac{1}{k+1}[,\,\, k \in \cN,\\
0, & \text{for } t=1.
\ea
\right.
\eenl
On the subintervals of $[1/2,1]$ where $u(\cdot)$ is absolutely continuous, one may use \eqref{forex} to compute $x(\cdot).$ On the other hand, one can easily check that
\benl
\begin{split}
x(1-1/k _+) &= x(1-1/k_-)e^2, \quad \text{if $k$ is odd},\\
x(1-1/k _+) &= x(1-1/k_-)e^{-2}, \quad \text{if $k$ is even},
\end{split}
\eenl
where $x(1-1/k_-)$ and $x(1-1/k _+)$ denote the left and right limits of $x$ at $t=1-1/k,$ respectively.
Hence, the p.d. solution $x(\cdot)$ of \eqref{eqex} associated with $u(\cdot)$ is given, for any $t\in [0,1],$ by
\benl
x(t):=
\left\{
\ba{cl}
\xb e^t, & \text{for } t\in[0,\frac{1}{2}[,\\
\xb e^{1/2} e^{-2}, & \text{for } t \in \bigcup_{k=1}^{\infty} [1-\frac{1}{2k},1-\frac{1}{2k+1}[,\\
\xb e^{1/2}, & \text{for } t \in \bigcup_{k=1}^{\infty} [1-\frac{1}{2k+1},1-\frac{1}{2k+2}[,\\
\xb e^{-1/2}, & \text{for } t=1.
\ea
\right.
\eenl
Notice that both $u$ and $x$ have infinitely many discontinuities and  unbounded variation, and are everywhere pointwise defined.
\end{example}

 \begin{theorem}[Dependence on the data]
 \label{cont-depL1}
 The following assertions hold.
\begin{itemize}
\item[{(i)}] For each $\xb \in \cR^{n}$ and $u\in \L^1 ([a,b];U)$ the function $v(\cdot)\mapsto x(\xb,u,v)(\cdot)$ is continuous from $L^1([a,b];V)$ to $L^1([a,b];\cR^n).$
\item[(ii)] For $r>0$ and $K\subseteq U$ compact, there exists a compact subset $K'\subset \cR^{n},$ such that the trajectories $x(\bar x,u,v)(\cdot)$ have values in $K',$ whenever we consider $\xb \in B_r(0),$  $u\in \L^1([a,b];K),$ and $v\in L^1([a,b];V).$
\item[(iii)] For each $r$ and $K$ as in  (ii), there exists a constant $M>0$ such that, for every $t \in[a,b],$ for all $\xb_1, \xb_2 \in B_r(0),$ for all $u_1,u_2 \in \L^1([a,b];K)$  and for every $v \in L^1([a,b];V),$ one has
\benl
\begin{split}
|x_1(t)-&x_2(t)| +  \|x_1-x_2\|_1 \\
\leq &\, M\Big[ |\xb_1-\xb_2|+ |u_1(a) - u_2(a)| \\
& + |u_1(t)-u_2(t)|+\|u_1-u_2\|_1
\Big].
\end{split}
\eenl
where  $x_1 := x(\xb_1,u_1,v_1),$ $x_2 := x(\xb_2,u_2,v_2).$
\end{itemize}

\end{theorem}

 A detailed proof of this result is available in \cite{AroRam13b}.

 \section{Proper extension of a standard minimum problem}\label{SecProper}

Let us consider  the (standard)  optimal control problem
 \bel{ottimoCed}
 \inf_{(u,v)\in AC\times L^1}\psi(x(b), u(b))\,,
\ee
where it is assumed that:
 \begin{itemize}
  \item[(i)] the cost map $\psi: \cR^{n+m}\to\cR $ is continuous;
  \item[(ii)] $ AC\times L^1$  stands for  $ AC([a,b ];U) \times L^1 ([a,b ];V);$  \item[(iii)] $x(\cdot)=x(\bar x,u,v)(\cdot),$ i.e. $x(\cdot)$ is the p.d. solution of the Cauchy problem \eqref{E}-\eqref{E0} where $\ub := u(a).$
 \end{itemize}

Our main concern here is to define  a {\it proper extension} of the minimum problem \eqref{ottimoCed}.

Let us give a formal notion of {\it proper extension}:

\begin{definition} Let $E$ be a set and let $\F:E\to\cR$ be a function. A  {\em proper extension} of a minimum problem
 \bel{original} \inf_{e\in E} \F(e)\eeq
 is  a new minimum  problem
   \bel{new} \inf_{{\hat e}\in {\hat E}} {\hat \F}({\hat e})\eeq
on a set $\hat E$ endowed with a limit notion and     such that there exists an injective map $i: E\to {\hat E}$
   verifying the following properties:
   \begin{itemize}
     \item[(i)] $\hat \F(i(e)) = \F(e)$ for all $e\in E$ and, moreover, for every ${\hat e}\in {\hat E}$ there exists a sequence $(e_k)$ in $E$ such that, setting ${\hat e}_k:= i(e_k)$, one has
         \bel{limite}
         \lim_{k \to \infty} \big(\hat e_k,\hat \F(\hat e_k) \big) = ({\hat e},\hat\F(
         \hat  e)),
         \ee
        \item[(ii)]  $\ds \inf_{e\in E} \F(e) =  \inf_{{\hat e}\in {\hat E} }{\hat \F}({\hat e}).$ \end{itemize}
     \end{definition}

After identifying $E$ and $\hat E$ with the set of pairs $(x(\cdot),u(\cdot))$ corresponding to controls in $AC \times L^1$ and  $\L^1 \times L^1,$ respectively,
we wish to investigate the question whether the  optimal control problem \bel{ottimoCed}
 \inf_{(u,v)\in \L^1\times L^1}\psi(x(b), u(b))\,,
\ee is a proper extension (with $i$ equal to the identity map) of the problem
\bel{ottimoC}
\inf_{(u,v)\in AC\times L^1}\psi(x(b), u(b)).
\ee

\begin{remark}
\label{rem-dense}
Notice that, in view of the definition of p.d. solution, the density property {(i)} is automatically satisfied.
\end{remark}

      To investigate the validity of  (ii), let us consider the reachable sets (at time $b$ for a fixed initial values $\xb$ and $\ub$):
\bel{reach1}
\begin{split}
\R:=  \{ (x,u)(b) :\,& (u,v) \in \L^1 \times L^1,\\
& u(a)=\ub,\ x = x(\bar x, u,v)\},
\end{split}
\ee
\bel{reach2}
\begin{split}
\R^+ := \{ (x,u)(b): \,& (u,v) \in AC \times L^1,\\
& u(a)=\ub,\  x = x(\bar x, u,v)\}.
\end{split}
\ee
Since the (Carath\'eodory) solution corresponding to an absolutely continuous $u$ is also a p.d. solution, one has
\bel{incl}
\R^+\subset \R.
\eeq
The inclusion is in general strict. However, the closure of the two sets always coincide.
\begin{theorem}\label{reach-th1}
\bel{reach-eq1}
\overline{{\R}} = \overline{{\R^+}}.
\ee
\end{theorem}

\begin{proof}
In view of \eqref{incl} it suffices to prove that $\overline{{\R}}  \subseteq \overline{{\R^+}}$. Assume by contradiction that there exists $y\in \overline{{\R}}$ such that
\bel{contr1}
d\Big(y,{\R^+}\Big)=\eta>0,\eeq
and let $\{(u_k,v_k)\} \subset \L^1([a,b];U)\times L^1([a,b];V)$ be a sequence of controls with $u_k(a)=\ub$  and such that the final points $y_k:= \Big(x(\bar x,u_k, v_k)(b)\,,\,u_k(b) \Big)$ verify
$d(y_k, y)\leq  \eta/3,$ for all $k\in \cN.$
Because of the definition of p.d. solution, for every $k\in\cN$ there exists  $(\hat u_k, v_k)\in AC([a,b ];U)\times L^1([a,b ];V)$ such that, setting $\hat y_k:= \Big(x(\bar x, \hat u_k,v_k)(b),\uh_k(b)\Big)$, one has $ d(\hat y_k ,y_k)\leq \eta/3,$
  so that
  $$
  d(y\, ,\, \hat y_k) \leq d(y\, ,\,  y_k) +  d(y_k , \hat y_k) \leq 2\eta/3,
  $$
  which contradicts \eqref{contr1}, as $\hat y_k\in {\R^+}.$
\end{proof}

Let us define the {\it value functions}
$$\begin{array}{c}
\vspace{2pt} \ds V_{AC}(\xb,\ub) :=  \inf_{(u,v)\in AC\times L^1}\psi(x(b), u(b))  \left(  = \inf_{\R^+} \psi(x,u)  \right),\\
\ds V_{\L^1} (\xb,\ub)  :=\inf_{(u,v)\in \L^1\times L^1}\psi(x(b), u(b)) \left(  = \inf_{\R} \psi(x,u) \right),
\end{array}$$
where it has been made explicit that these values depend on the initial data $(\xb,\ub).$

 \begin{corollary}
 \label{infug-th}
 For every  $(\bar x,\ub) \in \cR^{n+m},$ one has
\bel{infug}
V_{AC}(\bar x,\ub) =  V_{\L^1}(\bar x,\ub).
 \ee
  \end{corollary}

 Hence, also in view of Remark \ref{rem-dense}, we can conclude that problem \eqref{ottimoCed} {\em is a proper extension of   problem \eqref{ottimoC}}.

\section{Limits of minimun problems \\ with bounded variation}\label{SecBV}

Let us assume that $U$ is a convex set.

When the impulsive (possibly discontinuous) control $u$ has bounded total variation one can give a notion of solution based on the concept of  {\it graph completion} (see e.g. \cite{BreRam88}, \cite{Mil89}, \cite{SilVin96}).  This approach differs from the one above and can, in fact, be applied also to systems with no commutativity assumptions. However, if the commutativity hypothesis is standing, one can establish a one-to-one correspondence between the two concepts, as shown in Proposition \ref{ident} below.

For every $K\geq 0$, let us consider  the original system, supplemented with the variable $x_0=t,$
\bel{or+1}
\left\{
\ba{l}
\dot x_0 = 1,\\
\dot{x} = \tilde f(x,u,v)+ \sum_{\alpha=1}^m {\tilde g}_\alpha(x,u) \dot u^\alpha,\\
(x_0,x,u)(a) = (a,\bar x,\ub),
\ea
\right.
\ee
where the impulsive controls $u$  belong to the set
$$
BV_K([a,b];U) := \Big\{u: [a,b] \to U,\, {\rm Var}[u]\leq K\Big\},
$$
where ${\rm Var}[u]$ denotes the {\em total variation of $u.$} We also consider the subset
$$
AC_K([a,b];U):= AC([a,b];U)\cap BV_K([a,b];U).
$$
\if made of absolutely continuous maps $u(\cdot)$ with values in $U$ and such that $\ds \int_a^b |\dot u(t)|dt \leq K$.
\fi

\begin{definition}
We shall use $\U_K$ to denote the set  of maps $ (\ubf_0,\ubf)\in Lip( [0,1];[a,b]\times U)$ \footnote{\,Here $Lip( [0,1];[a,b]\times U)$ denotes the space of Lipschitz continuous function defined in $[0,1]$ and with values in $[a,b]\times U.$} such that, for a.a. $s\in [0,1]$, $\ubf_0'(s)\geq 0$, $\ubf_0'(s)+ |\ubf'(s)| \leq b-a+K,$  and, moreover, $ \ubf_0([0,1]) = [a,b]$.  These maps will be called {\em space-time controls with  variation not larger than $K$.}
 Furthermore,   $\U_K^+\subset \U_K$ will denote  the subset made of those space-time controls $(\ubf_0,\ubf)$ such that $\ubf_0'>0$ for a.a. $s\in [0,1]$.
 \end{definition}

 Let us consider the {\it space-time control system} in the interval $[0,1]$ given by
\bel{spacetime}
\left\{
\begin{array}{l}
\ybf_0' = \ubf_0',\\
 \ds\frac{d\ybf}{ds}= \ubf_0'\tilde f(\ybf,{\bf u},{\bf v})+ \sum_{\alpha=1}^m {\tilde g}_\alpha(\ybf,\ubf) {\ubf^\alpha}' , \\
(\ybf_0,\ybf,\ubf)(0) = (a,\bar x,\ub)\,,
\end{array}\right.
\ee
where the apex denotes differentiation with respect to the {\it pseudo-time} $s$,  $(\ubf_0,\ubf)\in \U_K$, and $\vbf \in L^1([0,1];V)$.
If   $\ubf$ is absolutely continuous,  \eqref{spacetime} can be regarded as an {\em ad hoc}  Lipschitz continuous  time-reparameterization of \eqref{or+1}, as it is made precise in the following statement (whose proof merely relies on the chain rule for derivatives).

\begin{proposition}
\label{ident}
Let us consider controls  $(u,v)\in AC_K([a,b];U)\times L^1([a,b];V)$  and an initial data $\bar x\in \cR^n.$  Let us set
$$
s(t) := \frac{\int_a^t (1 +|\dot u|) d\tau}{\int_a^b (1 +|\dot u|) d\tau}, \quad t(\cdot)=\ybf_0(\cdot) := s^{-1}(\cdot),
$$
and
$\ubf_0(s):= t(s)$, $\ubf(s) =  u\circ t(s)$, ${\bf v}(s):=  v\circ t(s).$
Then, $(\ubf_0,\ubf)\in \U_K^+$, $v\in L^1([0,1];V)$ and,  setting $x(\cdot):= x(\bar x, u,v)(\cdot) $, $\ybf(\cdot) =\ybf(\bar x, \ubf,{\bf v})(\cdot)$, one has
\bel{repar}
x\circ t(s)  = \ybf(s), \quad \forall s\in [0,1].
\ee
Conversely, if  $(\ubf_0,\ubf)\in \U_K^+$ , ${\bf v}\in L^1([0,1]; V)$,  setting
$
s(\cdot) := \ubf_0^{-1}(\cdot) $
and
$u(t) =  \ubf\circ s(t)$, $ v(t):=  {\bf v}\circ s(t)$,
one has that $(u,v) \in AC_K ([a,b];U)\times L^1([a,b];V)$ and
$$
x(t) = \ybf\circ s(t),  \qquad \forall t\in [a,b],
$$
where $x(\cdot):= x(\bar x, u,v)(\cdot) $, $\ybf(\cdot) =\ybf(\bar x, \ubf,{\bf v})(\cdot)$.
\end{proposition}

\vspace{1pt}

On the other hand, the space-time control system makes sense also when we allow $\ubf_0'(s)=0$  on some interval $[s_1,s_2]\subseteq[0,1]$. This accounts for a trajectory's { jump }  at  $t=\ubf_0(s_1) \big(=\ubf_0(s_2) \big).$ Notice that the trajectory `during' the jump is governed by the dynamics  $\sum_{\alpha=1}^m {\tilde g}_\alpha(\ybf,\ubf) {\ubf^\alpha}'.$ The commutativity hypothesis is here crucial, for it implies that the magnitude $\ybf(s_2)-\ybf(s_1)$ of the jump is independent of the path  $[s_1,s_2]\to\ubf(s).$

Consider now the {\em reachable sets} (at time $b$):
\bel{reach1K}
\R_K := \{ (x,u)(b) : (u,v)\in {BV}_K\times L^1\},
\ee
\bel{reach2K+}
\R_K^+ := \{ (x,u)(b) : (u,v)\in {AC}_K\times L^1\},
\ee
\bel{reach3}
\R_{K}^{BV} := \{ (\ybf(1),\ubf(1)) : \left((\ubf_0,\ubf),{\bf v}\right)\in \U_K\times L^1\},
\ee
\bel{reach4}
\R_{K}^{BV+} :=  \{ (\ybf(1),\ubf(1)) : \left((\ubf_0,\ubf),{\bf v}\right)\in \U_K^+\times L^1\},
\ee
where it is meant that the involved trajectories are the solutions of the corresponding Cauchy problems with given initial point $(\bar x,\ub).$

It follows easily that
\bel{reachinc}\begin{array}{c}
\qquad \R_{K}^{BV+}\subset \R_{K}^{BV}\,\,\forall K>0, \\\,\\
0 \leq K_1< K_2 \,\Rightarrow\,\R_{K_1}^{BV+} \subset \R_{K_2}^{BV+},\, \R_{K_1}^{BV} \subset \R_{K_2}^{BV}.
\end{array}
\ee
Moreover, in view of Proposition \ref{ident}, absolutely continuous solutions of \eqref{or+1} coincide with solutions of \eqref{spacetime} corresponding to $\U_K^+$, up to reparameterization. In particular,
\bel{reacheq}
\R^+_K  = \R_{K}^{BV+},\quad \text{for all } K\geq 0.
\ee
Furthermore,
\be
\R_K = \R_K^{BV}.
\ee
This identity can be verified by exploiting the commutativity assumption (Hypothesis \ref{AssComm}), which makes all the graph completions equivalent, and then by associating to each control $u \in BV$ its {\em rectilinear graph completion.} The latter is a Lipschitz continuous path in space-time obtained by bridging the discontinuities of $u$ by means of rectilinear segments.

One can also prove  (see \cite{MotRam96}) the following statement.

 \begin{proposition}\label{st-th1}
 For  every solution $\ybf(\bar x, \ubf_0,\ubf, {\bf v})$ corresponding to a control $(\ubf_0,\ubf)\in \U_K$ there exists a sequence $ \{({\ubf_0}_h,{\ubf}_h)\}_{h\in\cN}$ in $\U_K^+$ such that
 $$
 ({\ubf_0}_h,{\ubf}_h)\to (\ubf_0,\ubf), \quad \ybf(\bar x, {\ubf_0}_h,{\ubf}_h,  {\bf v}) \to  \ybf(\bar x, \ubf_0,\ubf,  {\bf v}),
 $$
 uniformly on $[0,1].$ In particular, one gets,
\bel{reachcleq1}
 \overline{\R_{K}^{BV+}} = \overline{\R_{K}^{BV}}.
\ee
\end{proposition}

\begin{remark}\label{st-th2}
(see \cite{BreRam88})
If the vector field $\tilde f$ is independent of the ordinary control $v$, then the set of solutions to \eqref{spacetime} corresponding to controls in $\U_K$ is closed in the $\C^0-$topology. In particular,
the reachable set $\R_{K}^{BV}$ is compact, so that
$$
\overline{\R_{K}^{BV}}= \R_{K}^{BV}.
$$
\end{remark}

\begin{remark}
Let us point out that relations \eqref{reachinc}, \eqref{reacheq}, Proposition \ref{st-th1} and Remark \ref{st-th2} are valid also in the case when the commutativity in Hypothesis \ref{AssComm} is not imposed.
\end{remark}

 \begin{theorem}
\label{reach-th}
\bel{reach}
\overline\R =\overline{ \bigcup_{K \geq 0}  \R_{K}^{BV}} = \overline{ \bigcup_{K \geq 0}  \R_{K}^{BV+}}.
\ee
\end{theorem}

We refer to \cite{AroRam13b} for a proof of the latter result.

Let us to consider the value functions corresponding to problems with bounded variation:
\benl
\ba{c}
V_{AC_K} (\xb,\ub)= \ds\inf_{(u,v)\in AC_K\times\L^1}\psi(x(b), u(b)),\\
V_{BV_K^+}(\xb,\ub)=\ds\inf_{(\ubf_0,\ubf,{\bf v})\in \U_K^+\times\L^1}\psi(\ybf(1), \ubf(1)),\\
V_{BV_K} (\xb,\ub) =\ds \inf_{(\ubf_0,\ubf,{\bf v})\in \U_K\times\L^1}\psi(\ybf(1), \ubf(1)).
\ea
\eenl

\begin{corollary}\label{inftyed-th}
For every $(\xb,\ub) \in \cR^n \times U,$ one has
\bel{infty}
 \lim_{K\to \infty} V_{BV_K} (\xb,\ub) = V(\xb,\ub).
\ee

\end{corollary}

\section{Considerations on dynamic programming}\label{SecLast}

For every $K \geq 0,$ let us consider the map $W_K: [a,b]\times \cR^M\times U\times [0,K]$ defined by letting $W_K(t,x,u,k)$ be the value function of the (impulsive)  minimum problem in $[t,b]$ with $u$-variation less than or equal to $K-k.$
By a reparameterization approach akin to the one in \cite{MotRam96}  one might prove that $W$
 is continuous and is the unique solution of a boundary value problem for a Hamilton-Jacobi equation involving the {\em compactified}  Hamiltonian
 \benl
 \ba{l}
 H(t,x,u,k,p_t,p_x,p_u,p_k) \\
  \ \ :=\sup_{w_0+ |w|\leq 1, v\in V} {\mathcal H}(t,x,u,k,p_t,p_x,p_u,p_k; w_0,w,v),
 \ea
  \eenl
 where the ${\mathcal H}$ is defined by
 $$
 \begin{array}{l}
 {\mathcal H}(t,x,u,k,p_t,p_x,p_u,p_k; w_0,w,v):=\\
 (p_t + p_x\cdot f(x,u,v)) w_0 + ( p_x \cdot g_\alpha + p_{u_\alpha}) w_\alpha + p_k|w_\alpha|.\end{array}$$

 Notice that $W_K(a,x,u,0)=V_{BV_K}(x,u)$, for all $(x,u)\in \cR^n\times{U}$. In particular  the Hamilton-Jacobi  equation
 $$
 H(t,x,u,k,\nabla W_K) = 0
 $$
 may be utilized  for both {sufficient conditions} of optimality and  { numerical analysis} of the problem with ${\rm Var} (u)\leq K.$ Via Corollary \ref{inftyed-th}, one can then address the general problem.

 \if{
 \subsection{Regularity of vector fields} Finally, let us comment the issue of the regularity of the data. The commutativity assumption $[g_\alpha, g_\beta]=0$ and the Simultaneous Flow-Box Theorem, which allow to define the coordinate change $\phi,$  have a classical interpretation for smooth vector fields. However, a notion of Lie bracket for locally Lipschitz continuous vector fields and a connected  Simultaneous Flow Box Theorem have been recently provided (see e.g.
 \cite{RamSus06}). Thus, one may conjecture that a notion of p.d. solution can be given also for the case when the $g_\alpha$'s are locally Lipschitz. Of course, one should somehow address the question raised by the non differentiability of the  transformation $\phi.$
 It is expectable that some notion of generalized differential has to be utilized. Accordingly, the vector field  $F$ would be  multivalued.
  }\fi

\bibliographystyle{plain}
\bibliography{impulsive}
\end{document}